\documentclass[11pt]{amsart}
\usepackage{amssymb}
\usepackage{amsfonts}
\usepackage{graphicx}
\usepackage{color}

\usepackage{xy}
\xyoption{all}

\usepackage{setspace}
\newtheorem{theorem}{Theorem}[section]
\newtheorem{proposition}[theorem]{Proposition}

\newtheorem{corollary}[theorem]{Corollary}

\theoremstyle{remark}
\newtheorem{remark}[theorem]{Remark}

\newtheorem{example}[theorem]{\bf Example}
\renewenvironment{proof}{{\noindent\bf Proof.}}{\hfill $\Box$\par\vskip3mm}

\newcommand{\Ker}{{\rm Ker}\,}

\newcommand{\Hom}{{\rm Hom}}
\newcommand{\End}{{\rm End}}

\newcommand{\Aa}{\mathcal{A}}
\newcommand{\Bb}{\mathcal{B}}
\newcommand{\Cc}{\mathcal{C}}

\newcommand{\Ff}{\mathcal{F}}

\newcommand{\Ii}{\mathcal{I}}

\newcommand{\Ll}{\mathcal{L}}
\newcommand{\Mm}{\mathcal{M}}

\newcommand{\Tt}{\mathcal{T}}
\newcommand{\Ss}{\mathcal{S}}

\def\NN{{\mathbb N}}

\def\AA{{\mathbb A}}
\def\KK{{\mathbb K}}

\begin{document}
\title[Triangular Matrix Coalgebras and Applications]{Triangular Matrix Coalgebras and Applications}

\begin{abstract}
We formally introduce and study generalized comatrix coalgebras and upper triangular comatrix coalgebras, which are not only a dualization but also an extension of classical generalized matrix algebras. We use these to answer several open questions on Noetherian and Artinian type notions in the theory of coalgebras, and to give complete connections between these. We also solve completely the so called finite splitting problem for coalgebras: we show that $C$ is a coalgebra such that the rational part of every left finitely generated $C^*$-module splits off if and only if $C=\left(\begin{array}{cc} D & M \\ 0 & E \end{array}\right)$ is an upper triangular matrix coalgebra, for a serial coalgebra $D$ whose Ext-quiver is a finite union of cycles, a finite dimensional coalgebra $E$ and a finite dimensional bicomodule $M$.
\end{abstract}

\author{Miodrag Cristian Iovanov\\}
\thanks{2010 \textit{Mathematics Subject Classification}. 15A30, 16T15, 18E40, 16T05}
\date{}
\keywords{comatrix coalgebra, triangular, Noetherian, torsion theory, splitting, rational module}
\maketitle

\section*{Introduction}



In classical ring theory and noncommutative algebra, triangular matrix rings represent a very important construction and tool. Given rings (or algebras) $A,B$ and an $A$-$B$-bimodule $M$, one can form the ring of upper triangular matrices $\left(\begin{array}{cc} A & M \\ 0 & B \end{array}\right)$. This was one of the first sources of rings that are Noetherian or Artinian on one side and not on the other. More generally one can consider generalized matrix rings $\left(\begin{array}{cc} A & M \\ N & B \end{array}\right)$, with $N$ a $B$-$A$-bimodule, and ring structure given via two bimodule maps $\varphi:M\otimes_B N\rightarrow A$ and $N\otimes_AM\rightarrow B$, satisfying certain compatibility conditions. These rings are associated with Morita contexts (and in one to one bijection with them). ``Higher dimensional" versions of these rings are also very useful tools. 

\vspace{.3cm}

In this paper, we formally introduce and study generalized comatrix coalgebras, and triangular comatrix coalgebras, and use these to answer several open questions for coalgebras. The construction is natural to consider as a dual notion to generalized matrix rings; in fact, although not specifically written down anywhere, this construction is probably known for some time. For example, constructions involving certain coalgebras similar to triangular comatrix coalgebras were already used in \cite{I2,CI} to give examples and counterexamples in coalgebras, or in connection to duality problems. The dual algebras of such coalgebras are algebras which are generalized or upper triangular matrix rings, respectively. In fact, the (triangular) comatrix coalgebras can be also seen as generalization of the matrix rings. This is because a finite dimensional algebra $A$ which is a generalized matrix ring is always the dual of a generalized comatrix coalgebra $C$, and its category of left modules is the category of right comodules of $C$.

\vspace{.3cm}

We give the relation of generalized comatrix coalgebras to existing Morita theory for coalgebras (see \cite{Ber}; also \cite{Tak2}) in Section \ref{sec.1}. Some brief details are included for fixing notation of comatrix theory and for sake of completeness. In Section \ref{sec.n}, we use these generalized comatrix coalgebras to provide various examples and counterexamples concerning several important notions in coalgebra theory. Such are Artinian coalgebras and comodules, co-Noetherian comodules, coalgebras $C$ such that the dual algebra $C^*$ is Noetherian or ``almost" Noetherian, quasi-finite and strictly quasifinite coalgebras, finitely cogenerated comodules. Recall that a left comodule is {\it Artinian} if it is Artinian in the usual way, (or equivalently, as a right $C^*$-module). A coalgebra is left Artinian if it is Artinian as a left comodule (equivalently, $C^*$ is a left Noetherian ring). A (left or right) comodule $M$ is {\it finitely cogenerated} if it embeds $C^n$ for some positive integer $n$. A comodule $M$ is {\it co-Noetherian} if $M/N$ is finitely cogenerated for every subcomodule $N$ of $M$. The comodule $M$ is said to be {\it quasi-finite} if $\Hom^C(F,M)$ is finite dimensional for every finite dimensional comodule $N$, and it is called {\it strictly quasi-finite} if $M/N$ is quasi-finite for every subcomodule $N$ of $M$. 

\vspace{.3cm}

We recall the known connections between these notions in Section \ref{sec.n} and give some new ones, as well as provide examples to show all the existing inclusions of these classes are strict. For instance, we show that a comodule $M$ is Artinian if and only if $M^*$ is Noetherian, and that this is further equivalent to every $C^*$-submodule of $M^*$ being closed. This implies a Matlis type duality for the category of finitely generated left $C^*$-modules for a coalgebra $C$ such that $C^*$ is left Noetherian. We give equivalent conditions for a triangular comatrix coalgebra to be Artinian, dualizing (and extending) a result from triangular matrix rings. Using this, we construct an example of a left Artinian but not right Artinian coalgebra; its dual will be a ring which is left but not right Noetherian. We also investigate the inclusion between important classes of comodules $\{{\rm Artinian \, comodules}\} \subseteq \{{\rm co-Noetherian \, comodules}\}\subseteq \{{\rm strictly \, quasi-finite \, comodules}\}$, and show that all inclusions are strict. In fact, with the aid of triangular comatrix coalgebras we construct an example of a coalgebra which is left co-Noetherian but not left Artinian, left but not right co-Noetherian, and left but not right strictly quasi-finite, thus showing the importance of the concept.

\vspace{.3cm}

In the last section, we use triangular comatrix coalgebras and other techniques to answer another categorical flavored question studied in literature. Given a coalgebra $C$, the category of $C$-comodules $\Mm^C$ embeds in that of left $C^*$-modules ${}_{C^*}\Mm$. In general, given an abelian category $\Aa$ and a subcategory $\Cc$, one can ask the question of whether every object $X$ of $\Aa$ or in a certain subclass decomposes as a direct sum of its $\Cc$-torsion part $T(X)$ (= the sum of all subobjects of $X$ which belong to $\Cc$), and a complement. This is called the splitting problem, and was the subject of many investigations; such is, for example, the so called singular torsion, or classical torsion for modules over integral domains \cite{T1,T2,T3,Rot}. In general, one hopes to better understand the objects for which such a splitting occurs. The above mentioned embedding of $C$-comodules into $C^*$-modules provides another natural framework for this problem. This was studied in several places \cite{C,I1,IO,NT}; in \cite{NT,I1} it is shown that if all $C^*$-modules split, then $C$ must be finite dimensional. The situation for which only finitely generated $C^*$-modules split (the finite splitting property) is studied in \cite{C,IO}. If $C$ is cocommutative, and the rational part of every $C^*$-module splits off, then $C$ is a finite direct sum of serial and finite dimensional coalgebras \cite{C}. If $C^*$ is local and has this Rat-splitting property, then $C$ must be uniserial, and $C^*$ is a DVR; on the other hand, serial coalgebras do have the splitting property for f.g. modules \cite{IO}. However, the general situation remained open. We give the complete answer here, and classify coalgebras $C$ which have this splitting property for the rational part of every finitely generated left module over $C^*$ in Section 3. Our approach is self-contained and does not make use of previously known results for the colocal or cocommutative case, but only uses a few simple observations from \cite{IO}. Namely, we show that a coalgebra for which the rational part of every finitely generated left $C^*$-modules splits off must be a generalized triangular comatrix coalgebra $\left(\begin{array}{cc} D & M \\ 0 & E \end{array}\right)$, where $D$ is a serial coalgebra whose Ext-quiver is a finite disjoint union of cycles (equivalently, it is serial and almost connected), $E$ is a finite dimensional coalgebra and $M$ is a $D$-$E$-bicomodule (Theorem \ref{t.SplT}).\\ 
This will allow us to note again an interesting left-right broken symmetry: it is possible that $C$ has the finite splitting property for left $C^*$-modules, but not for the right $C^*$-modules. We construct many examples of combinatorial flavor in Section 4. The dual algebras of coalgebras in the class classified here include classical examples, such as (matrix rings containing) the algebra of formal power series, or complete path or monomial algebras of cyclic quivers, or more generally, of quivers whose cycles satisfy certain easy combinatorial conditions (Example \ref{e.4.6}).  

\section{Generalized Comatrix Coalgebras}\label{sec.1}

For the general theory of coalgebras and their comodules we refer to the classical textbooks \cite{A,DNR,M,Sw}. Often, for a vector space $M$, we will need to consider the finite topology on $M^*$ with a basis of neighborhoods of $0$ consisting of submodules of $M^*$ of the type $X^\perp=\{f\in M^*|f\vert_X=0\}$, for finite dimensional subcomodules of $M$. Any topological references will be understood with respect to this topology, unless otherwise specified. 

We begin the study with a procedure which allows us to obtain a class of examples which will be central to the classification and our results. It is also a construction that will provide many interesting examples. First recall that in general, if $A,B$ are rings and $P\in{}_A\Mm_B$ is an $A$-$B$-bimodule, $Q\in{}_B\Mm_A$ is a $B$-$A$-bimodule, 
$\varphi^*:P\otimes_BQ \rightarrow A$ - which we denote shortly $\varphi^*(p\otimes_Bq)=pq$ - is a morphism of $A$-bimodules, and $\psi^*:Q\otimes_AP\rightarrow B$ - which we denote shortly $\psi^*(q\otimes_Ap)=qp$ - is a morphism of $B$-bimodules, such that $(A,B,P,Q,\varphi^*,\psi^*)$ is a Morita context, then one can define the general matrix ring $\left(\begin{array}{cc} A & P \\ Q & B  \end{array} \right)$ with multiplication 
$$\left(\begin{array}{cc} a & p \\ q & b  \end{array} \right)\cdot\left(\begin{array}{cc} a' & p' \\ q' & b' \end{array} \right)=\left(\begin{array}{cc} aa'+pq' & ap'+pb' \\ qa'+bq' & bb'+qp' \end{array} \right).$$ 
In fact, the conditions of $(A,B,X,Y,\varphi^*,\psi^*)$ being a Morita context are equivalent to the associativity of this multiplication.

Now let us briefly present the dual construction for coalgebras, which is a natural dualization of that for rings. In fact, the two are essentially identical as particularizations of a more general construction in monoidal categories (see Remark \ref{r.1} below). We include some details to establish notation and for sake of completeness. Let $C,D$ be coalgebras, $X$ a $D$-$E$-bicomodule and $Y$ an $E$-$D$-bicomodule. Let us write $X\ni x\mapsto x_{-1}\otimes x_0\otimes x_1 \in D\otimes X\otimes E$ for the bicomodule structure of $X$ and $Y\ni y\mapsto y_{-1}\otimes y_0\otimes y_1\in E\otimes Y\otimes D$ for the bicomodule structure of $Y$; there will be no danger of confusion in the following construction. Also, let us consider two morphisms,  $\varphi:D\rightarrow X\square_EY$ of $D$-$D$-bicomodules and $\psi:E\rightarrow Y\square_DX$ of $E$-$E$ bicomodules, where $\square$ is the cotensor product of comodules (e.g. see \cite{DNR}); write $\varphi(d)=(\sum\limits_d)d^X\square_E d^Y$ (a symbolic sum) and $\psi(e)=(\sum\limits_e)e^Y\square_De^X$. Let $C$ be the generalized co-matrix coalgebra $C=\left(\begin{array}{cc} D & X \\ Y & E \end{array}\right)=D\oplus X\oplus Y\oplus E$ with comultiplication and counit given by

\begin{eqnarray*}
\Delta\left(\begin{array}{cc} d & x \\ y & e \end{array} \right) & = &
\left(\begin{array}{cc} d_1 & 0 \\ 0 & 0 \end{array} \right)\otimes \left(\begin{array}{cc} d_2 & 0 \\ 0 & 0  \end{array} \right) +
\left(\begin{array}{cc} 0 & d^X \\ 0 & 0 \end{array} \right)\otimes \left(\begin{array}{cc} 0 & 0 \\ d^Y & 0  \end{array} \right) + \\
&  & \left(\begin{array}{cc} x_{-1} & 0 \\ 0 & 0 \end{array} \right)\otimes \left(\begin{array}{cc} 0 & x_{0} \\ 0 & 0 \end{array} \right) + 
\left(\begin{array}{cc} 0 & x_{0} \\ 0 & 0 \end{array} \right)\otimes \left(\begin{array}{cc} 0 & 0 \\ 0 & x_{1} \end{array} \right) +\\
& & \left(\begin{array}{cc} 0 & 0 \\ y_0 & 0 \end{array} \right)\otimes \left(\begin{array}{cc} y_1 & 0 \\ 0 & 0 \end{array}\right) +
\left(\begin{array}{cc} 0 & 0 \\ 0 & y_{-1} \end{array} \right)\otimes \left(\begin{array}{cc} 0 & 0 \\ y_0 & 0 \end{array} \right) +\\
& & \left(\begin{array}{cc} 0 & 0 \\ e^Y & 0 \end{array} \right)\otimes \left(\begin{array}{cc} 0 & e^X \\ 0 & 0  \end{array} \right)+
\left(\begin{array}{cc} 0 & 0 \\ 0 & e_1 \end{array} \right)\otimes \left(\begin{array}{cc} 0 & 0 \\ 0 & e_2  \end{array} \right)\\
\varepsilon\left(\begin{array}{cc} d & x \\ y & e \end{array} \right) & = &\varepsilon_D(d)+\varepsilon_E(e)
\end{eqnarray*}

It is straightforward to check the fact that the comultiplication of $C$ is coassociative is equivalent to the fact that $(D,E,X,Y,\varphi,\psi)$ is a Morita-Takeuchi context \cite{Ber}. The conditions from \cite[Section 3, p.874]{Ber} on elements read here $x_0\otimes (x_1)^Y\otimes (x_1)^X=(x_{-1})^X\otimes (x_{-1})^Y\otimes x_0$ for $x\in X$ and $y_0\otimes (y_1)^X\otimes (y_2)^Y=(y_{-1})^Y\otimes y_{-1}^X\otimes y_0$ for $y\in Y$. Moreover, for example, when $X$ and $Y$ are finite dimensional, the dual ring is isomorphic to the classical generalized matrix ring associated to the Morita context $(D^*,E^*,X^*,Y^*,\varphi^*,\psi^*)$: $C^*=\left(\begin{array}{cc} D & X \\ Y & E \end{array} \right)^*\cong \left(\begin{array}{cc} D^* & X^* \\ Y^* & E^* \end{array} \right)$, where since $X\in{}^D\Mm^E$ and $Y\in{}^E\Mm^D$ we have $X\in{}_{E^*}\Mm_{D^*}$ so $X^*\in{}_{D^*}\Mm_{E^*}$, and similarly $Y\in{}_{D^*}\Mm_{E^*}$ so $Y^*\in{}_{E^*}\Mm_{D^*}$. Also, it is easy to generalize the construction to $n$ coalgebras $(C_i)_i$ and $C_i$-$C_j$-comodules $M_{ij}$ with $M_{ii}=C_i$, together with morphisms of $C_i$-$C_j$-bicomodules $\phi_{ij}:M_{ij}\rightarrow M_{ik}\square_{C_k}M_{kj}$ ($k\notin \{i,j\}$), satisfying appropriate conditions so that the comultiplication and counit induced similarly on the generalized comatrix coalgebra
$$
\left(\begin{array}{cccc}
C_1 & M_{12} & \dots & M_{1n} \\
M_{21} & C_2 & \dots & M_{nn} \\
\dots & \dots & \dots & \dots \\
M_{n1} & M_{n2} & \dots & C_{n}
\end{array}\right)
$$
are coassociative and counital.

First we observe a coalgebra decomposition analogue to the algebra situation. Although it is not required in its full generality for the purpose of the study of the splitting property (nor is the full generalized comatrix coalgebra), we include it again for sake of completeness. Recall from \cite{Rad2} that for an idempotent $e$ of $C^*$, the subspace $eCe$ of $C^*$ has a coalgebra structure given by $ece\longmapsto ec_1e\otimes ec_2e$ and counit equal to the restriction of the counit $\varepsilon$ of $C$ to $eCe$. Moreover, we have that the dual ring is $(eCe)^*\cong eC^*e$. Similarly, we note that if $e,f$ are idempotents of $C^*$ with $e+f=\varepsilon={\mathbf 1}_{C^*}$, then $eCf$ has a $fCf$-$eCe$ bicomodule structure given by $(ecf)_{(-1)}\otimes (ecf)_{(0)}\otimes (ecf)_{(1)}=fc_1f\otimes ec_2f\otimes ec_3e$ (this defines both the left and right comodule structures in such a way that they are compatible). Hence $eCf$ has a naturally induced structure of an $(eCe)^*$-$(fCf)^*$-bimodule. If we make the identifications $eC^*e=(eCe)^*$ and $fC^*f=(fCf)^*$, then this $(eCe)^*$-$(fCf)^*$-bimodule structure of $eCf$ becomes the usual $eC^*e$-$fC^*f$-bimodule structure $ec^*e\cdot exf\cdot fd^*f=ec^*exfd^*f$. Testing the bicomodule structure is a computation similar to that of \cite{Rad2}, and for the last part, we observe that
\begin{eqnarray*}
ec^*e\cdot exf\cdot fd^*f & = & (fd^*f)((exf)_{(-1)})(exf)_{(0)}(ec^*e)((exf)_{(1)})=\\
& = & (fd^*f)(fx_1f)[ex_2f](ec^*e)(ex_3e)\\
& = & f(x_1)f(x_2)d^*(x_3)f(x_4)f(x_5)f(x_6)x_7e(x_8)e(x_9)e(x_{10})\cdot \\
&  & \;\;\;\;\;\;\;\;\;\;\;\;\;\;\;\;\;\;\;\;\;\;\;\;\;\;\;\;\;\;\;\;\;\;\;\;\;\;c^*(x_{11})e(x_{12})e(x_{13})=\\
& = & f(x_1)d^*(x_2)f(x_3)x_4e(x_5)c^*(x_6)e(x_7) \,\,\,\, (e^2=e; f^2=f)\\
& = & (ec^*e)x(fd^*f)
\end{eqnarray*}
Moreover, we have bicolinear maps $\varphi:eCe\ni ece\longmapsto fc_1e\otimes ec_2f\in (fCe)\square_{fCf}(eCf)$ and $\psi:fCf\ni fcf\longmapsto ec_1f\otimes fc_2e\in (eCf)\square_{eCe}(fCe)$, such that $(eCe,fCf, fCe, eCf,\varphi,\psi)$ is a Morita-Takeuchi context. We have:

\begin{proposition}\label{p.eCf}
Let $C$ be a coalgebra and $C=X\oplus Y$ as left $C$-comodules. Let $e=\varepsilon\vert_X$ and $f=\varepsilon\vert_Y=\varepsilon-e$. Then $X=eC$, $Y=fC$ and 
$${C\cong\left(\begin{array}{cc}
	eCe & fCe \\
	eCf & fCf
\end{array}
\right)
}$$
as coalgebras.
\end{proposition}
\begin{proof}
It is a straightforward copy and dualization of the well known algebra situation. One only needs to observe that $C=eCe\oplus fCe\oplus eCf\oplus fCf$ as vector spaces, and the isomorphism $C\ni c\leftrightarrow \left(\begin{array}{cc} ece & fce \\ ecf & fcf \end{array} \right)\in \left(\begin{array}{cc} eCe & fCe \\ 	eCf & fCf \end{array} \right)$ is then proven to be an isomorphism of coalgebras, using the above definition of a generalized comatrix coalgebra. Indeed, this follows noting that 
\begin{eqnarray*}
c_1\otimes c_2 & = & (ec_1e+ec_1f+fc_1e+fc_1f)\otimes(ec_2e+ec_2f+fc_2e+fc_2f)=\\
& = & ec_1e\otimes ec_2e+fc_1e\otimes ec_2f+fc_1f\otimes fc_2e+fc_1e\otimes ec_2e +\\
& & ec_1f\otimes ec_2e+fc_1f\otimes ec_2e+ec_1f\otimes fc_2e+fc_1f\otimes fc_2f
\end{eqnarray*}
These terms correspond precisely to the terms in the definition of the comultiplication of the generalized comatrix coalgebra above.
\end{proof}

Now we restrict the general comatrix coalgebra construction to a more special situation. Let $D,E$ be coalgebras and let $M$ be a $D-E$ bicomodule. We use Sweedler's convention with the summation symbol omitted. Write $m\mapsto m_{(0)}\otimes m_{(1)}$ for the right $E$-coaction on $M$ and $m\mapsto m_{[-1]}\otimes m_{[0]}$ for the left $D$-coaction on $M$. We consider the coalgebra of upper triangular co-matrices $\left(\begin{array}{cc} D & M \\ 0 & E \end{array}\right)=D\oplus M\oplus E$, with comultiplication $\Delta$ and counit $\varepsilon$ given by 
\begin{eqnarray*}
\Delta\left(\begin{array}{cc} d & m \\ 0 & e \end{array} \right) & = &
\left(\begin{array}{cc} d_1 & 0 \\ 0 & 0 \end{array} \right)\otimes \left(\begin{array}{cc} d_2 & 0 \\ 0 & 0  \end{array} \right) +
\left(\begin{array}{cc} m_{[-1]} & 0 \\ 0 & 0 \end{array} \right)\otimes \left(\begin{array}{cc} 0 & m_{[0]} \\ 0 & 0 \end{array} \right) + \\
& = & \left(\begin{array}{cc} 0 & m_{[0]} \\ 0 & 0 \end{array} \right)\otimes \left(\begin{array}{cc} 0 & 0 \\ 0 & m_{[1]} \end{array} \right) +
\left(\begin{array}{cc} 0 & 0 \\ 0 & e_1 \end{array} \right)\otimes \left(\begin{array}{cc} 0 & 0 \\ 0 & e_2  \end{array} \right)\\
\varepsilon\left(\begin{array}{cc} d & m \\ 0 & e \end{array} \right) & = &\varepsilon_D(d)+\varepsilon_E(e)
\end{eqnarray*}
These formulas give a coalgebra structure. In fact, the dual algebra is isomorphic to the upper triangular matrix algebra (ring) $\left(\begin{array}{cc} D & M \\ 0 & E \end{array} \right)^*\cong \left(\begin{array}{cc} D^* & M^* \\ 0 & E^* \end{array} \right)$, where $M\in{}^D\Mm^E$ implies $M\in{}_{E^*}\Mm_{D^*}$ so $M^*\in{}_{D^*}\Mm_{E^*}$. In general, if $A,B$ are rings and $P\in{}_A\Mm_B$ is an $A$-$B$-bimodule, then the upper triangular matrix ring is $\left(\begin{array}{cc} A & P \\ 0 & B \end{array} \right)=A\oplus M\oplus B$, with multiplication $\left(\begin{array}{cc} a & p \\ 0 & b  \end{array} \right)\cdot\left(\begin{array}{cc} a' & p' \\ 0 & b' \end{array} \right)=\left(\begin{array}{cc} aa' & ap'+pb' \\ 0 & bb' \end{array} \right)$ (the Morita context has the connecting maps equal to $0$). Let us note that these coalgebras appear in a common situation, which is in duality with a well known algebra fact:

\begin{proposition}\label{p.triang}
Let $C$ be a coalgebra, and $C=X\oplus Y$ as left $C$-comodules. Then $X$ is a subcoalgebra if and only if $\Hom^C(X,Y)=0$. Moreover, in this situation, if $e\in C^*$ is such that $e\vert_X=\varepsilon\vert_X$ and $e\vert_Y=0$, and $f=\varepsilon-e$ so $f\vert_Y=\varepsilon\vert_Y$, $f\vert_X=0$, then $X=eC=eCe$, $Y=fC$ and $C\cong\left(\begin{array}{cc} eCe & fCe \\ 0 & fCf \end{array}\right)$ as coalgebras.
\end{proposition}
\begin{proof}
Since $X$ is left $C$-subcomodule, it is a subcoalgebra if and only if it is also a right subcomodule, so a left $C^*$-submodule. That means that $c^*\cdot X\subseteq X$ for all $c^*\in C^*$. But since all morphisms of right $C^*$-modules (left $C$-comodules) $f:C\rightarrow C$ are of the form $f(x)=c^*\cdot x$ for some $c^*\in C^*$, we get that $X$ is a subcoalgebra if and only if $f(X)\subseteq X$ for all $f\in\End({}^CC)$. This means $\Hom^C(X,C)=\Hom^C(X,X)$, but since $\Hom^C(X,C)=\Hom^C(X,X)\oplus\Hom^C(X,Y)$, the condition is equivalent to $\Hom^C(X,Y)=0$. The last assertion follows since if $c\in C$ then $x=ec\in X$ so $ecf=xf=f(x_1)x_2=0$ because $x_1\otimes x_2\in X\otimes X$ since $X$ is a subcoalgebra, and $f\vert_X=0$. Hence $eCf=0$, and we can use the above Proposition \ref{p.eCf}.
\end{proof}

\begin{example}
Any semitrivial extension is a quotient of a generalized triangular comatrix coalgebra. Indeed, if $D$ is a coalgebra and $M$ is a $D$-bicomodule, the semitrivial extension $D\oplus M$ with comultiplication $(d,m)\mapsto (d_1,0)\otimes(d_2,0)+(m_{-1},0)\otimes(0,m_0)+(0,m_0)\otimes(m_1,0)$ and counit $\varepsilon(d,m)=
\varepsilon_D(d)$ can be seen as the quotient of $\left(\begin{array}{cc} D & M \\ 0 & D\end{array}\right)$ by the coideal $I=\{\left(\begin{array}{cc} d & 0 \\ 0 & -d\end{array}\right)|d\in D\}$. 
\end{example}

We now recall another well known fact. A left module $H$ over a generalized triangular matrix ring $\left(\begin{array}{cc} A & P \\ 0 & B \end{array} \right)$ is given by the following data: $X\in{}_A\Mm$ and $Y\in{}_B\Mm$ and $\varphi:P\rightarrow \Hom(Y,X)$ a morphism of $A-B$-bimodules, equivalently, a morphism of left $A$-modules $\psi:P\otimes_BY\rightarrow X$, and written $p\cdot y=\varphi(p)(y)=\psi(p\otimes_B y)$, such that $H=\left(\begin{array}{c}  X  \\  Y  \end{array} \right)=X\oplus Y$ with action $\left(\begin{array}{cc} a & p \\ 0 & b \end{array} \right)\cdot \left(\begin{array}{c} x  \\ y  \end{array} \right)=\left(\begin{array}{cc} ax+p\cdot y  \\ by  \end{array} \right)$. 

\begin{remark}\label{r.1}
We note that the above well known construction of triangular matrix rings is valid in any monoidal abelian category, that is, given $\Cc$ a monoidal abelian category, $A,B$ two algebras in $\Cc$ and $P$ an $A$-$B$-bimodule, we can form the generalized upper triangular algebra $H$ in $\Cc$ as before by letting $H=A\oplus P\oplus B$ with multiplication equal to the sum of all the structure maps $A\otimes A\rightarrow A$, $A\otimes P\rightarrow P$, $P\otimes B\rightarrow B$ and $B\otimes B\rightarrow B$ and unit given by the canonical map $({\mathbf 1}_\Cc\rightarrow A\oplus B)$ obtained from the units of $A$ and $B$. Similarly, the modules over $H$ have the same description as above as $X\oplus Y$ with $X$ a left $A$-module in $\Cc$ and $Y$ a left $B$-module in $\Cc$, and $\psi:P\otimes_BY\rightarrow X$ a morphism of left $A$-modules. Using this for the monoidal category ${}_K\Mm^{o}$ - the dual category to that of vector spaces - we get that generalized triangular comatrix coalgebras $C$ are in fact generalized triangular matrix algebras in ${}_K\Mm^{o}$ and consequently, left comodules over $H=\left(\begin{array}{cc} D & M \\ 0 & E \end{array}\right)$ are of the form $\left(\begin{array}{cc} X \\ Y \end{array}\right)$ with $X\in {}^D\Mm$, $Y\in {}^E\Mm$ and together with a morphism of left $D$-comodules $\psi:X\rightarrow M\square_EY$. This can of course be also determined by a direct computation parallel to that from the algebra case, and so we leave these details to the reader. However, this fact will not be necessary in what follows. 
\end{remark}

In the situation when $A=D^*$, $B=E^*$ and $P=M^*$, and $H=\left(\begin{array}{c}  X  \\  Y  \end{array} \right)$ is a left module over the upper triangular matrix algebra $\left(\begin{array}{cc} A & P \\ 0 & B \end{array} \right)$, let us define the $P$-rational part of $Y$ to be $Rat_P(Y)=\{y\in Y|\,\exists \sum\limits_kz_k\otimes m_k\in X\otimes M,{\rm\,such\,that\,}p\cdot y=\sum\limits_{k}p(m_k)z_k,\,\forall p\in P\}$. 

\begin{proposition}\label{p1}
If $C=\left(\begin{array}{cc} D & M \\ 0 & E \end{array} \right)$ is a coalgebra of upper triangular co-matrices and $C^*=\left(\begin{array}{cc} A & P \\ 0 & B \end{array} \right)$, then for any left $C^*$-module $H=\left(\begin{array}{c} X  \\  Y  \end{array} \right)$ we have
$$Rat_{C^*}\left(\begin{array}{cc}  X  \\  Y  \end{array} \right)=\left(\begin{array}{c} Rat_A(X)   \\ Rat_B(Y)\cap Rat_P(Y)   \end{array} \right)$$
\end{proposition}
\begin{proof}
Assume that $\left(\begin{array}{c} x \\  y  \end{array} \right)\in Rat_{C^*}\left(\begin{array}{cc}  X  \\  Y  \end{array} \right)$ so there are $x_i\in X$, $y_i\in Y$, $d_i\in D$, $m_i\in M$, $e_i\in E$ such that for any $c^*=\left(\begin{array}{cc} a & p \\ 0 & b \end{array} \right)\in C^*$
\begin{eqnarray}
\left(\begin{array}{cc} a & p \\ 0 & b \end{array} \right)\cdot \left(\begin{array}{c} x \\ y  \end{array} \right)=\sum\limits_i<\left(\begin{array}{cc} a & p \\ 0 & b \end{array} \right);\left(\begin{array}{cc} d_i & m_i \\ 0 & e_i \end{array} \right)>\left(\begin{array}{c} x_i  \\ y_i  \end{array} \right)
\end{eqnarray}
(we write $<c^*,c>=c^*(c)$, for $c^*\in C^*$, $c\in C$). In particular, for $b=0$, $p=0$ and $a\in A$ we get $ax=\sum\limits_i a(d_i)x_i$ so $x\in Rat_A(X)$; for $a=0$, $p=0$ and $b\in B$ we get $by=\sum\limits_{i}b(e_i)y_i$ i.e. $y\in Rat_B(Y)$; when $a=0$, $b=0$ and $p\in P$ we get $p\cdot y=\sum\limits_ip(m_i)\otimes x_i$ so $y\in Rat_P(Y)$.\\
Conversely, let $x\in Rat_A(X)$ and $y\in Rat_B(Y)\cap Rat_P(Y)$, so there are $\sum\limits_ix_i\otimes d_i\in X\otimes D$, $\sum\limits_j y_j\otimes e_j\in Y\otimes E$, $\sum\limits_k z_k\otimes m_k\in X\otimes M$ such that $ax=\sum\limits_ia(d_i)x_i$, $by=\sum\limits_jb(e_j)y_j$, $p\cdot y=\sum\limits_kp(m_k)z_k$ for all $a\in A,b\in B, p\in P$. Then we see that 
\begin{eqnarray*}
\left(\begin{array}{cc} a & p \\ 0 & b \end{array} \right)\cdot \left(\begin{array}{c} x  \\  y  \end{array} \right) =
\left(\begin{array}{cc} ax + py \\  by \end{array} \right)
 & = & \sum\limits_i<\left(\begin{array}{cc} a & p \\ 0 & b \end{array} \right);\left(\begin{array}{cc} d_i & 0 \\ 0 & 0 \end{array} \right)>\left(\begin{array}{c} x_i  \\ 0   \end{array} \right) \\ 
\sum\limits_k<\left(\begin{array}{cc} a & p \\ 0 & b \end{array} \right);\left(\begin{array}{cc} 0 & m_k \\ 0 & 0 \end{array} \right)>\left(\begin{array}{c} z_k  \\ 0  \end{array} \right)
& + &
\sum\limits_j<\left(\begin{array}{cc} a & p \\ 0 & b \end{array} \right);\left(\begin{array}{cc} 0 & 0 \\ 0 & e_j \end{array} \right)>\left(\begin{array}{cc}  0  \\ y_j   \end{array} \right)
\end{eqnarray*}
which shows that $\left(\begin{array}{c} x  \\ y  \end{array} \right)\in Rat_{C^*}\left(\begin{array}{c} X   \\  Y  \end{array} \right)$
\end{proof}

\begin{remark}\label{r1}
If $P=M^*$ is finite dimensional then $Rat_P(Y)=Y$. Indeed, for $y\in Y$ let us look at the morphism of left $A$-modules $\psi:P\rightarrow X$, $\psi(p)=p\cdot y$. Let $p_k=m_k^*\in P$ be a basis of $P$ and $z_k=\psi(p_k)$. Then any $p\in P$ is written $p=\sum\limits_kp(m_k)p_k$ (dual bases $(p_k),(m_k)$) and so $p\cdot y=\psi(p)=\sum\limits_{k}p(m_k)\psi(p_k)=\sum\limits_kp(m_k)z_k$ i.e. $y\in Rat_P(Y)$. 
\end{remark}

A coalgebra $C$ is said to have the {\it left finite Rat-splitting property}, or the {\it left f.g. Rat-splitting property} (or shortly $C$ is {\it left finite Rat-splitting}) if the Rational torsion of every finitely generated left $C^*$-module $M$ splits off in $M$. The generalized triangular comatrix coalgebra offers a way to produce a large class of examples:

\begin{proposition}\label{p.TrSp}
Let $D$ be a coalgebra with the left finite Rat-splitting property, $E$ a finite dimensional coalgebra and $M$ a finite dimensional $D$-$E$-comodule. Then the upper triangular comatrix coalgebra $C=\left(\begin{array}{cc} D & M  \\ 0 & E \end{array} \right)$ has the left finite Rat-splitting property.
\end{proposition}
\begin{proof}
Let $H=\left(\begin{array}{c}  X  \\  Y  \end{array} \right)$ be a finitely generated left $C^*$-module, and $\left(\begin{array}{c} x_i  \\ y_i  \end{array} \right)_{i}$ a finite system of generators, so for any $\left(\begin{array}{c}  x  \\  y  \end{array} \right)\in H$ there is a linear combination $\left(\begin{array}{c}  x \\ y   \end{array} \right)=\sum\limits_i\left(\begin{array}{cc} a_i & p_i \\ 0 & b_i \end{array} \right)\cdot \left(\begin{array}{c} x_i  \\ y_i  \end{array} \right)$. First we note that for $x=0$, this formula shows that any $y\in Y$ can be written as $y=\sum\limits_{i}b_iy_i$ so $Y$ is finitely generated over $B=E^*$ by the $y_i$'s. Consequently, $Y$ is finite dimensional since $B$ is so. Fixing $y=0$, for every $x$ we get an expression $x=\sum\limits_ia_ix_i+p_i\cdot y_i$. Since $M^*$ is finitely generated as left $D^*$-module, we can write $p_i=\sum\limits_ja_{ij}q_j$, $a_{ij}\in D^*$ with $q_j\in M^*$ fixed generators, and so $x=\sum\limits_ia_ix_i+\sum\limits_{i,j}(a_{ij}q_j)\cdot y_i=\sum\limits_ia_ix_i+\sum\limits_{i,j}a_{ij}(q_j\cdot y_i)$. Hence $X$ is finitely generated by the $x_i$'s and the $(q_j\cdot y_i)$'s. Thus $X=Rat_A(X)\oplus X'$ as left $A$-modules, and $Y=Rat_B(Y)=Rat_P(Y)$, because both $B=E^*$ and $P=M^*$ are finite dimensional. Thus $Rat_{C^*}\left(\begin{array}{c} X  \\  Y  \end{array} \right)=\left(\begin{array}{c} Rat_A(X)   \\ Rat_B(Y)   \end{array} \right)$ by Proposition \ref{p1}. But now it is easy to note that $\left(\begin{array}{c}  X' \\  0  \end{array} \right)$ is a left $C^*$-submodule of $H$. Therefore 
$$H=\left(\begin{array}{cc} X  \\  Y  \end{array} \right)=\left(\begin{array}{cc} Rat_A(X)  \\  Rat_B(Y)  \end{array} \right)\oplus \left(\begin{array}{cc}  X' \\  0  \end{array} \right)$$
i.e. $Rat_{C^*}(H)$ splits off in $H$.
\end{proof}

\section{Noetherian and Artinian Objects}\label{sec.n}

We use the constructions of the previous Section to give applications to several important notions in coalgebra theory, related to Noetherian and Artinian objects. In fact, the situation when the dual algebra $C^*$ of an algebra $C$ is Noetherian will be important in our study. Hence we include a brief study of Artinian $C$-comodules and Noetherian $C^*$-modules. A coalgebra is said to be called \emph{left (respectively, right) Artinian} if it is Artinian as a left (respectively, right) $C$-comodule. We recall that a coalgebra is left Artinian if and only if $C$ is Artinian as a right (respectively left) $C^*$-module, equivalently, $C^*$ is a left (respectively, right) Noetherian algebra. We first note that we have an extended version of \cite[Proposition 2.5]{IO}.

\begin{proposition}
Let $C$ be a left Artinian coalgebra (so $C^*$ is left Noetherian). Then for every left (and also every right) $C^*$-module $M$, $Rat(M)$ is the sum of all finite dimensional submodules of $M$, $Rat(M)=\{x\in M|C^*\cdot x{\rm\,is\,finite\,dimensional}\}$.
\end{proposition}
\begin{proof}
Since $C^*$ is left Noetherian, it is also left almost Noetherian (i.e. cofinite ideals are finitely generated), which implies that $C$ is coreflexive (see \cite{Rad}), that is, every cofinite ideal $I$ is closed, i.e. it is the $\perp$ of some $X\subset C$, $I=X^\perp$. Hence every finite dimensional $C$-comodule (left or right) is rational, and so the conclusion follows.
\end{proof}



A right $C$-comodule $M$ is called {\it finitely cogenerated} if it embeds in a coproduct of finitely many copies of $C$. We note:

\begin{proposition}\label{p.fcgen}
A right $C$-comodule $M$ is finitely cogenerated if and only if $M^*$ is a finitely generated $C^*$-module.
\end{proposition}
\begin{proof}
If $M\hookrightarrow C^n$ is a monomorphism, dualizing we get an epimorphism $(C^*)^n\rightarrow M^*\rightarrow 0$. Conversely, let $f_1,\dots,f_n$ generate the right $C^*$-module $M^*$. Let $u_i:M\rightarrow C$ be defined by $u_i(m)=f_i(m_0)m_1$. It is well known and easy to see that $u_i$ are morphisms of right comodules, and so $M_i=\Ker(u_i)$ are subcomodules of $M$. Thus, we have embeddings $M/M_i\hookrightarrow C$. Also, $\bigcap\limits M_i=0$. Indeed, if $x\in\bigcap\limits_iM_i$, then for every $f\in M^*$, we have $f=\sum\limits_if_ic^*_i$, $c^*_i\in C^*$ and so $f(x)=\sum\limits_i(f_ic^*_i)(x)=\sum\limits_if_i(c^*_i\cdot x)=\sum\limits_if_i(x_0)c_i(x_1)=0$, because $x\in \Ker(f_i),\,\forall i$. So $x=0$ because $f(x)=0$, for all $f\in M^*$. Hence, we have a monomorphism of right $C$-comodules
$$M=\frac{M}{\bigcap\limits_iM_i}\hookrightarrow \bigoplus\limits_{i=1}^n\frac{M}{M_i}\hookrightarrow C^n$$  
\end{proof}

The equivalence of (ii) and (iii) in the following result is known from \cite{HR74}; also, it is part of \cite[Lemma 3.2]{NT}. A particular version of this proposition for coalgebras is proved in \cite[Lemma 3.4]{HIT}. We provide here a general proof for comodules, generalizing all the above.

\begin{proposition}\label{p.Noeth}
The following are equivalent for a left $C$-comodule $M$.\\
(i) Every submodule of $M^*$ is closed (in the finite topology of $M^*$). \\
(ii) $M^*$ is Noetherian.\\
(iii) $M$ is Artinian. 
\end{proposition}
\begin{proof}
(ii)$\Rightarrow$(i) is easy, since any finitely generated left submodule of $M^*$ is closed (this is a well known fact; one can also see \cite[Lemma 1.1]{I}).\\
(i)$\Rightarrow$(iii) Let $(X_n)_n$ be a descending chain of subcomodules of $M$; then $N=\sum\limits_nX_n^\perp\subseteq M^*$ is closed, so $\sum\limits_nX_n^\perp=X^\perp$, where $X$ is a subcomodule in $M$ since $N$ is a $C^*$-submodule of $M^*$. Note first that $X\subseteq\bigcap\limits_nX_n$. Assuming that the chain $X_1\supseteq X_2\supseteq\dots$ does not terminate, we can choose $x_n\in X_n\setminus X_{n+1}$, and we can easily see that the sum of vector spaces $\sum\limits_nKx_n+X$ is direct. This allows us to get an $f\in C^*$ such that $f(x_n)=1$, for all $n$ and $f\vert_X=0$ (completed suitably to a linear function on $C$). But then $f\in X^\perp$ and $f\notin X_n^\perp=I_n$ for all $n$, so $f\notin\bigcup\limits_nX_n^\perp=\sum\limits_nX_n^\perp=N$, a contradiction. Therefore $X_n=X_{n+1}=...$ from some $n$ onwards. \\
(iii)$\Rightarrow$(ii) Let $P\subseteq M^*$ be a submodule and assume it is not finitely generated; then there is a sequence of submodules $P_n=\sum\limits_{i=1}^nC^*a_i$ of $M^*$ with $a_{n+1}\notin P_n$ for all $n$, so $P_1\subsetneq P_2\subsetneq\dots\subsetneq P_n\subsetneq\dots$ is an ascending chain of submodules. They are closed (since they are finitely generated), so $P_n=X_n^\perp$ with $X_n=P_n^\perp$, and we have a strictly descending chain $X_1\supsetneq X_2\supsetneq\dots\supsetneq X_n\supsetneq\dots$ of left subcomodules of $M$ (since if $X_n=X_{n+1}$ then $P_n=X_n^\perp=X_{n+1}^\perp=P_{n+1}$) which is a contradiction.
\end{proof}

Note that this proposition shows that a pseudocompact $C^*$-module is Noetherian in the category of left pseudocompact $C^*$-modules, if and only if it is Noetherian as a $C^*$-module (recall that the category of pseudocompact left $C^*$ modules is the dual of the category of left $C$-comodules; see \cite[Section 2.6]{DNR}). We also have

\begin{corollary}\label{c.duality}
If $C$ is a left Artinian coalgebra (equivalently, $C^*$ is a left Noetherian algebra), then the categories of Noetherian left $C^*$-modules $C^*{\rm-Noeth}$ (or finitely generated left modules) and left Artinian $C$-comodules $C{\rm-Art}$ (or finitely cogenerated left $C$-comodules) are in duality.
\end{corollary}
\begin{proof}
Let $F:C{\rm-Art}\longrightarrow C^*{\rm-Noeth}$ be the functor $F(X)=X^*$. It is easy to see that it is faithful. If $M$ is a Noetherian $C^*$-module, then $M\cong (C^*)^n/N=(C^n)^*/N$. But since $C$ is Artinian as a left comodule, so is $C^n$ and therefore submodules of $(C^n)^*$ are closed (and finitely generated). Hence $N=X^\perp$, $X\subseteq C^n$. Hence, $M\cong (C^n)^*/X^\perp \cong X^*$ as left $C^*$-modules. The isomorphism is standard, as there is an exact sequence $0\rightarrow X^\perp\rightarrow (C^n)^*\rightarrow X^*\rightarrow 0$. Moreover, $F$ is also full. If $f:Y^*\rightarrow X^*$ is a morphism of Noetherian left $C^*$-modules, then for any cofinite open $N=Z^\perp\subseteq X^*$, with $Z$ a finite dimensional subcomodule of $X$, we have that $f^{-1}(N)$ is cofinite in $Y^*$ and it is also closed since $Y^*$ is Noetherian. Thus, $f^{-1}(N)=U^\perp$ with $U$ a finite dimensional subcomodule of $Y$, hence $f^{-1}(N)$ is open in $Y^*$. Therefore, $f$ is a continuous morphism of pseudocompact modules, and so $f=h^*$, $h:X\rightarrow Y$ a morphism of left $C$-comodules (by the duality between pseudocompact $C^*$-modules and $C$-comodules). So $F$ is (contravariant) full, faithful and surjective on objects.
\end{proof}

We will use this duality in our investigations. Using this and a well known result of rings, we also get an interesting:

\begin{proposition}
Let $D,E$ be coalgebras and $M$ a $D$-$E$-bicomodule. Then the coalgebra $C=\left(\begin{array}{cc} D & M \\ 0 & E\end{array}\right)$ is left (respectively, right) Artinian if and only if $D,E$ are left Artinian coalgebras and $M$ is a finitely cogenerated $D$ comodule (respectively, a finitely cogenerated right $E$-comodule).
\end{proposition}
\begin{proof}
Using the propositions above, $C$ is left Artinian if and only if $C^*=\left(\begin{array}{cc} D^* & M^* \\ 0 & E^*\end{array}\right)$ is left Noetherian, which, by a well known result of rings, is equivalent to $D^*$ and $E^*$ being left Noetherian and $M^*$ being finitely generated as a left $D^*$-module. Again by the previous results, this proves the statement. 
\end{proof}

\begin{example}
This allows us to obtain coalgebras $C$ which are left and not right Artinian. Indeed, it suffices to take $D$ an infinite dimensional Artinian coalgebra (for example, the divided power coalgebra whose dual algebra is the ring of formal power series $K[[X]]$), $E=K$ and $M=D$ (with right $E$-comodule structure being that of a $K$-vector space). Then, as above, $\left(\begin{array}{cc} D & M \\ 0 & E\end{array}\right)$ is left and not right Artinian. 
\end{example}

\subsection*{Related notions}

We use this opportunity to provide here some examples and results regarding several important notions considered in literature in coalgebra theory (e.g. see \cite{CNO, CGT, NT3, Rad, Rad2, Tak2, W1, WZ}).\\
We recall from \cite{W1,WZ} that a left $C$ comodule $M$ is called \emph{co-Noetherian} if and only if for every subcomodule $N$ of $M$, the comodule $M/N$ is finitely cogenerated. A left $C$-comodule $M$ is called {\it quasi-finite} \cite{Tak2} if $\Hom^C(F,M)$ is finite dimensional for every finite dimensional (equivalently, every simple) left $C$-comodule $F$. A comodule $M$ is called {\it strictly quasi-finite} \cite{NT3} if every quotient comodule is quasifinite. We note an interesting characterization of co-Noetherian comodules which parallels that of Artinian comodules.

\begin{proposition}
A left $C$-comodule $M$ is co-Noetherian if and only if every closed submodule of $M^*$ is finitely generated.
\end{proposition}
\begin{proof}
For every subcomodule $N$ of $M$, we have an isomorphism of left $C^*$-modules $N^\perp\cong (M/N)^*$. Hence, $N^\perp\cong (M/N)^*$ is finitely generated if and only if $M/N$ is finitely cogenerated by Proposition \ref{p.fcgen}. This implies the statement.
\end{proof}

\begin{corollary}
If a left $C$-comodule $M$ is Artinian, then $M$ is co-Noetherian.
\end{corollary}
\begin{proof}
It follows since when $M$ is Artinian, $M/N$ is Artinian too so it has finite dimensional socle. Hence, it embeds in a finite coproduct power of $C$ (as $C$ is an injective cogenerator of $\Mm^C$).
\end{proof}

We have thus inclusions of classes
$$\{{\rm Artinian\,comodules}\}\subset \{{\rm co-Noetherian\,comodules}\}\subset\{{\rm strictly\,quasi-finite\,modules}\}$$
the last one being proved in \cite{NT3}. It is also shown there that the second inclusion is strict (\cite[Example 1.5]{NT3}), and it is also remarked that the inclusion between Artinian comodules and quasi-finite comodules is also strict (\cite[Remark 1.7]{NT3}). The three classes coincide over an almost connected coalgebra (i.e. a coalgebra whose coradical is finite dimensional; see \cite[Proposition 1.6]{NT3}). The following example shows that the first inclusion is also strict, and in fact it gives an example of a coalgebra which is left co-Noetherian but not left Artinian (i.e. as left comodule). It also shows that a coalgebra can be left co-Noetherian but not right co-Noetherian, and left strictly quasi-finite but not right strictly quasi-finite (another example of right but not left strictly quasi-finite coalgebra is contained in \cite[Example 2.7]{NT3}).

\begin{example}
Let $C$ be the coalgebra with basis $\Bb=\{a,x_n,b_n|n\geq 1\}$ and comultiplication $\Delta(a)=a\otimes a$, $\Delta(b_n)=b_n\otimes b_n$ and $\Delta(x_n)=b_n\otimes x_n+x_n\otimes a$, $\varepsilon(a)=\varepsilon(b_n)=1$ and $\varepsilon(x_n)=0$. Note that this coalgebra can be understood also as the coalgebra of upper triangular co-matrices 
$$\left(\begin{array}{cc} D & M \\ 0 & E \end{array}\right)$$ 
with $D$ the coalgebra with basis $\{b_n|n\geq 1\}$, $E$ the coalgebra of basis $\{a\}$ and $M$ the $D-E$ bicomodule with basis $\{x_n|n\geq 1\}$ and left comultiplication $x_n\longmapsto b_n\otimes x_n$ and right comultiplication $x_n\longmapsto x_n\otimes a$. Note also that it can also be thought as the path coalgebra of the following quiver:
$$
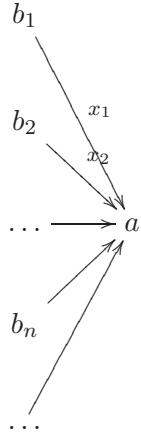
\xymatrix{
b_1\ar[ddr]^{x_1} & \\
b_2\ar[dr]^{x_2} & \\
\dots \ar[r] & a \\
b_n \ar[ur] & \\
\dots \ar[uur] &
}$$
We see that the left injective indecomposable comodules are $E_l(a)=<a,x_n|n\geq 1>$ and $E_l(b_n)=<b_n>$ - which are simple comodules - where $<X>$ is the space spanned by the set $X$. Since the socle of $C$ is semisimple infinite dimensional, it follows that $C$ is not Artinian as a left comodule. Let $N$ be a left subcomodule of $C$. Since $<b_n>$ are injective, we have that the part of the socle of $N$ contained in $\bigoplus\limits_n<b_n>$ splits off, so $N=X\oplus\bigoplus\limits_{n\in F}<b_n>$, with $F\subseteq \NN$ (note that all subcomodules of $\bigoplus\limits_{n\in \NN}<b_n>$ are of the form $\bigoplus\limits_{n\in F}<b_n>$ for some $F\subseteq \NN$). Thus $X\cap\bigoplus\limits_{n\in\NN}<b_n>=0$, and so there is a monomorphism $X\hookrightarrow E_l(a)$, and we can write $C=E'\oplus \bigoplus\limits_{n\in\NN}<b_n>$, with $X\subseteq E'\cong E_l(a)$. If $X=0$, then we see that $C/N\cong E_l(a)\oplus\sum\limits_{n\notin F}<b_n>$ which embeds in $C$. When $X\neq 0$, then the socle of $X$ is $<a>$ and then $E'\cong E_l(a)$ and $E'/X$ is a isomorphic to a quotient of $E_l(a)/<a>\cong\bigoplus\limits_n<b_n>$. Hence, $E'/X$ contains at most one of each type of simple comodule $<b_n>$. Therefore, since $C/N\cong E'/X\oplus \bigoplus\limits_{n\notin F}<b_n>$, it follows that $C/N$ is semisimple and contains at most two summands of each type of simple left $C$-comodule, hence, embeds in $C^2$. Thus, $C$ is co-Noetherian as a left comodule, and hence, also left strictly quasi-finite. \\
Note also that as a right comodule, the injective envelopes of the simples are $E_r(<a>)=<a>$ and $E_r(<b_n>)=<b_n,x_n>$, and if $P=\bigoplus\limits_{n\in \NN}<b_n>$ we see that $C/P\cong \bigoplus\limits_{n\in\NN}<a>$, and so $\Hom^C(<a>^C,C/P)$ is not finite dimensional. Therefore, $C$ is not right strictly quasifinite. Hence, $C$ is neither right co-Noetherian nor right Artinian. 
\end{example}

\section{The Finite Splitting Property}\label{sec.s}

Let $C$ be a coalgebra. Let $\Ss$ be a set of representatives of simple left $C$-comodules, and $\Tt$ a set of representatives of simple right $C$-comodules. We fix some notations: we have $C=\bigoplus\limits_{S\in \Ss}E(S)^{n(S)}$ as left $C$-comodules, where $E(S)$ is an injective hull of $S$ contained in $C$, and $n(S)$ is the multiplicity of $S$ in $C_0$ - the coradical of $C$. Similarly, $C=\bigoplus\limits_{T\in\Tt}E(T)^{p(T)}$ as right $C$-comodules. We recall some facts from \cite{IO} with respect to the finite Rat-splitting property of coalgebras. We call a left $C$-comodule $M$ almost finite if it has only finite dimensional proper subcomodules. We recall the following result which combines several results from \cite[Section 1 \& 2]{IO}.

\begin{theorem}\label{t.1}
Let $C$ be a coalgebra which has the left finite Rat-splitting property, that is, the rational part of every left $C^*$-module splits off. Then we have:\\
(i) The injective envelope of every left $C$-comodule is almost finite and countable dimensional.\\
(ii) $C$ is left Artinian, injective as a left $C^*$-module, and $C^*$ is left Noetherian.\\
(iii) $C_0$ is finite dimensional (i.e. $C$ is almost connected), equivalently, $\Ss$ is finite (and so is $\Tt$). 
\end{theorem}

Let us now assume that $C$ has the left finite Rat-splitting property. We can divide left simple $C$-comodules into two sets: $\Ss=\Ii\sqcup\Ff$, with $\Ii=\{S\in \Ss|E(S){\rm\,is\,infinite\,dimensional}\}$, $\Ff=\Ss\setminus\Ff$. Let $C_\Ii=\bigoplus\limits_{S\in\Ii}E(S)^{n(S)}$ and $C_\Ff=\bigoplus\limits_{S\in\Ff}E(S)^{n(S)}$. Then

\begin{proposition}\label{p.tr}
Assume that the coalgebra $C$ has the left finite Rat-splitting property. Under the notation introduced before, $\Hom^C(C_\Ii,C_\Ff)=0$. Moreover, $\Hom^C(C_\Ff,C_\Ii)$, $\Hom^C(C_\Ff,C_\Ff)$ and $\Hom^C(C_\Ff,C)$ are finite dimensional. Consequently, $C_\Ii$ is a subcoalgebra, and $C$ is isomorphic to the generalized triangular matrix coalgebra $\left(\begin{array}{cc} C_\Ii & fCe \\ 0 & fCf \end{array}\right)$, where $e,f$ are as in Proposition \ref{p.triang} such that $C_\Ii=eC$, $C_\Ff=fC$ and furthermore, $fCe$ and $fCf$ are finite dimensional.
\end{proposition}
\begin{proof}
$\Hom^C(C_\Ii,C_\Ff)=\bigoplus\limits_{S\in\Ii; L\in\Ff}\Hom(E(S),E(L))^{(n(S)n(L))}=0$, since for $S\in\Ss$ and $L\in\Ll$ we have $\Hom(E(S),E(L))=0$. This is because since $E(S)$ is infinite dimensional with only finite dimensional subcomodules, it cannot have any finite dimensional quotients other than $0$; thus since any $h:E(S)\rightarrow E(L)$ in ${}^C\Mm$ has finite dimensional image, $h=0$. The last assertion follows from the isomorphism of vector spaces $\Hom^C(C_\Ff,C_\Ff)\oplus \Hom(C_\Ff,C_\Ii)=\Hom^C(C_\Ff,C)\cong (C_\Ff)^*$, and by the fact that $C_\Ff$ is finite dimensional ($\Ff$ is finite).\\
For the last part, we apply Proposition \ref{p.triang} and note that $eCf=0$ so $C_\Ii=eC=eCe+eCf=eCe$; also, we note that $C_\Ff=fC=fCe\oplus fCf$ is finite dimensional, so $fCe$ and $fCf$ are finite dimensional. 
\end{proof}

Thus, the situation of Proposition \ref{p.TrSp} is typical for coalgebras with the left finite Rat-splitting property. By Proposition \cite[2.9]{IO}, the coalgebra $C_\Ii$ has the left finite Rat-splitting property as a subcoalgebra of $C$. So we now concentrate on characterizing coalgebras $C$ with the left finite Rat-splitting, and such that all injective indecomposable left comodules are infinite dimensional. 

\begin{proposition}\label{p.inj}
Let $C$ be a coalgebra with the left finite Rat-splitting property, and let $M$ be a left $C$-comodule which is finitely cogenerated and has no finite dimensional non-trivial quotients. Then $M$ is an injective $C$-comodule.
\end{proposition}
\begin{proof}
Since $C^*$ is left Noetherian, by the duality from Corollary \ref{c.duality} we get that $M^*$ is a finitely generated left $C^*$-module and it has no non-zero finite dimensional submodules. We show that any exact sequence $0\rightarrow M\rightarrow X\rightarrow N\rightarrow 0$ in ${}^C\Mm$ with  $N$ a finite dimensional left $C$-comodule must split. Indeed, consider such a sequence. We have that $0\rightarrow N^*\rightarrow X^*\rightarrow M^*\rightarrow 0$ is an exact sequence in ${}_{C^*}\Mm$ and $X^*$ is finitely generated (Noetherian) too. We note that in this situation $Rat(X^*)=N^*$. Indeed, assume otherwise. Then there would be a finite dimensional submodule $H\subseteq X^*$ with $N^*\subsetneq H$. But this means that $H/N^*$ is a nonzero finite dimensional left $C^*$-module which embeds in $M^*$, and this is a contradiction. Then, $Rat(X^*)=N^*$ shows that the sequence $0\rightarrow N^*\rightarrow X^*\rightarrow M^*\rightarrow 0$ must split, and again by the aforementioned duality, the sequence $0\rightarrow M\rightarrow X\rightarrow N\rightarrow 0$ splits. \\
Now let $M\subseteq X$ be a monomorphism from $M$ to an arbitrary left comodule $X$. The set $\{Y\subseteq X|M\cap Y=0\}$ is inductive, so let $Y$ be a maximal element (by Zorn's Lemma). We claim that $M\oplus Y=X$. Assume otherwise; then there is some $x\in X\setminus (M+Y)$, and let $R=x\cdot C^*$ which is finite dimensional, since it is rational. We have a monomorphism $\varphi:M\cong \frac{M\oplus Y}{Y}\hookrightarrow \frac{M\oplus Y+R}{Y}$, whose cokernel $M+Y+R/M+Y=R/(M+Y)\cap R$ is finite dimensional. Thus, by the first part, $\varphi$ splits. This means that there is $Z\subseteq M+Y+R$ such that $Z\cap (M+Y)=Y$ and $M+Y+Z=M+Y+R$. In particular, we get that $M\cap Z\subseteq M\cap(M+Y)\cap Z= M\cap Y=0$; but $Z\supseteq Y$, and by the maximality of $Y$, we get $Z=Y$. In particular, we get $M+Y=M+Y+Z=M+Y+R$, which shows that $x\in R\subseteq M+Y$, a contradiction. Hence $M$ splits off in $X$, which concludes the proof since $X$ is an arbitrary comodule.
\end{proof}

We recall a few definitions used in \cite{C,CGT,GTN08,IO,LS}; the terminology is derived from serial modules. Given a coalgebra $C$, a $C$-comodule $M$ is called uniserial if $M$ has a unique composition series, equivalently, the lattice of submodules of $M$ is totally ordered. $M$ is called serial if $M$ is a direct sum of uniserial comodules. A coalgebra is called left (right) serial if it is serial as a left (right) comodule. It is called serial if it is serial as left and right comodule (see \cite{CGT}). 

\begin{corollary}\label{c.1}
Let $C$ be a coalgebra with the left finite Rat-splitting property. Then $C_\Ii$ is left serial.
\end{corollary}
\begin{proof}
Since $C_\Ii$ is the direct sum of the infinite dimensional injective indecomposable left $C$-comodules,we need to show that $E(S)$ is serial for each $S\in\Ii$. By the results of \cite[Section 1]{CGT} or \cite[Proposition 3.2]{IO}, it is enough to show that $E(S)/S$ has simple socle. But $E(S)/S$ has only finite dimensional proper subcomodules by Theorem \ref{t.1}, and since it is infinite dimensional, it follows that it has no finite dimensional nontrivial quotients and that it is Artinian. Moreover, the socle of $E(S)/S$ is finite dimensional so it is finitely cogenerated; therefore, by Proposition \ref{p.inj}, $E(S)/S$ is injective. But note that it is in fact indecomposable: if $E(T)/S=M\oplus N$, then at least one of $M$ and $N$ is infinite dimensional and not proper. Since $E(S)/S$ is injective, we deduce that it has simple socle.
\end{proof}

\begin{proposition}
Let $C$ be a coalgebra with the left finite Rat-splitting property and such that $C=C_\Ii$, that is, any indecomposable injective left $C$-comodule is infinite dimensional. If $S$ is a simple left $C$-comodule, then $Ext^C(S,S')\neq 0$ for at most one $S'\in\Ss$. 
\end{proposition}
\begin{proof}
Assume otherwise. This means that there are indecomposable comodules $M_1,M_2$ with $0\rightarrow S_i\rightarrow M_i\rightarrow S\rightarrow 0$, so since $C$ is left serial by Corollary \ref{c.1}, we get that $E(S_1)/S_1\cong E(S)\cong E(S_2)/S_2$ for $S_1,S_2\in\Ss$ and $S_1\not\cong S_2$. Let $M$ be the pull-back of the diagram
$$\xymatrix{ & E(S_1)\ar[dr]^{p_1} & \\
M\ar[ur]^{u_1}\ar[dr]_{u_2} & & E(S)\\
& E(S_2)\ar[ur]_{p_2} &}$$
so $M=\{(x_1,x_2)|p_1(x_1)=p_2(x_2)\}\subset E(S_1)\oplus E(S_2)$, where $u_1(x_1,x_2)=x_1$,  $u_2(x_1,x_2)=x_2$. Note that $S_1\oplus S_2\subseteq M$, and since $M\subseteq E(S_1)\oplus E(S_2)$, we get that $S_1\oplus S_2$ is the socle of $M$. 
We also see that if $(x_1,0)\in M$, then $p_1(x_1)=p_2(0)=0$ so $x_1\in S_1$. Similarly, if $(0,x_2)\in M$, then $x_2\in S_2$. This shows that $\ker(u_1)=S_2$ and $\ker(u_2)=S_1$. As $u_1,u_2$ are surjective (since $p_1$ and $p_2$ are), we conclude $M/S_1\cong E(S_2)$ and $M/S_2\cong E(S_1)$. Now this implies that $M$ is infinite dimensional and furthermore, it has only finite dimensional proper subcomodules. Indeed, let $0\neq X\subsetneq M$. Then $X$ contains a simple subcomodule of $M$. But since $S_1\not\cong S_2$, and $S_1\oplus S_2$ is the socle of $M$, we see that $S_1$ and $S_2$ are the \emph{only} simple subcomodules of $M$. Hence, we must have either $S_1\subseteq X$ or $S_2\subseteq X$. If, for example, $S_1\subseteq X$ then $X/S_1$ embeds in $M/S_1\cong E(S_2)$, and it is not equal to $M/S_1$ (since $X\neq M$). It follows that $X/S_1$ is finite dimensional, and so is $X$.\\
Finally, since $M$ is infinite dimensional with only finite dimensional proper subcomodules, it follows that $M$ has no finite dimensional non-trivial quotients. Since $M$ has finite dimensional socle $S_1\oplus S_2$, it embeds in $C\oplus C$, so it is also finitely cogenerated, and therefore $M$ is injective by Proposition \ref{p.inj}. Since the socle of $M$ is $S_1\oplus S_2$, we must have $M\cong X_1\oplus X_2$, with $S_1\subset X_1\cong E(S_1)$  and $S_2\subset X_2\cong E(S_2)$. But this shows that $E(S_2)\cong M/S_1\cong X_1/S_1\oplus X_2$, and since $E(S_2)$ is indecomposable, we get $X_1/S_1=0$, which is impossible (since $X_1\cong E(S_1)\supsetneq S_1$). 
\end{proof}

\begin{proposition}\label{p.3.6}
Let $C$ be a coalgebra with the left finite Rat-splitting property. Then the Ext-quiver of $C_\Ii$ is a finite disjoint union of cycles $S_1\rightarrow S_2\rightarrow\dots\rightarrow S_k\rightarrow S_1$.
\end{proposition}
\begin{proof}
By the previous section, we have that there is exactly one arrow starting at any vertex $S\in\Ss$: indeed, we know that there is at most one arrow out of $S$, and $S\neq E(S)$ so there is some nontrivial extension of $S$. Also, by the previous Proposition, we also have that there is at most one arrow into $S$. But it is easy to see that a finite graph $\Ss$ which has the property that exactly one arrow goes out of any vertex and at most one comes into any vertex is a disjoint union of cycles. The finiteness follows from Theorem \ref{t.1}(iii).
\end{proof}

Recall that for a comodule $M$, its coradical filtration $(M_n)_n$ is build up as the Loewy series of $M$, $M_n=L_n(M)$, where $L_0(M)$ is the socle of $M$, and $L_{n+1}(M)=L_0(M/L_n(M))$. Also, one conveys to write $L_{-1}(M)=0$. Since a comodule $M$ is a chain (uniserial) if and only if its coradical filtration is a composition series (possibly infinite), and it is the only composition series in this case, it makes sense to consider the simple left comodules $S$ which appear as factors $M_n/M_{n-1}$ of this composition series in $0\subset M_0\subset M_1\subset\dots$ of $M$ and also consider their multiplicities $[M:S]$, which could be infinite. We convey to write $M_{-1}=0$. Let us also write $[M_0,M_1/M_0,M_2/M_1,\dots,M_n/M_{n-1},\dots]$ for the {\it ordered sequence} of the factors in the composition series of $M$; this is well determined (up to isomorphism) since there is only one such composition series.

\begin{proposition}\label{p.3.7}
Let $C$ be a coalgebra with the left finite Rat-splitting property, and such that every left injective indecomposable comodule is infinite dimensional. Then $C$ is a serial coalgebra.
\end{proposition}
\begin{proof}
By Corollary \ref{c.1} and the hypothesis that $C=C_\Ii$, we already know that $C$ is left serial. Let $T$ be a simple right $C$-comodule and let $S=T^*$. The previous proposition shows that the sequence of factors of the left injective indecomposable $C$-comodule $E(S)$ is of the form $[\underbrace{S_1,S_2,...,S_n},\underbrace{S_1,S_2,...,S_n},\dots]$ (where $S=S_1$). This shows that there is a left comodule $X$ of arbitrarily large finite length which has $S=S_1$ on the last position in its sequence of factors: $[\dots\underbrace{S_1,...,S_n},S_1]$. Dualizing, we get that $X^*$ is a uniserial module whose socle is $T$, and which can be made to have arbitrary (Loewy) length. Consequently, the injective hull $E(T)$ of $T$ is infinite dimensional and has infinite Loewy series. For each $S\in\Ss$, let $\alpha(S)\in\Ss$ be the end of the arrow starting at $S$ in the Ext quiver of (left comodules for) $C$, that is, we have $soc(E(S)/S)\cong \alpha(S)$. We note that the composition series of $E(S)$ is then $[S,\alpha(S),\alpha^2(S),\dots,\alpha^m(S),\dots]$. By Proposition \ref{p.3.6}, $\alpha$ is a permutation of the finite set $\Ss$. Let $N>1$ be such that $\alpha^N={\rm id}_{\Ss}$. We get then for each $k$ that $\alpha^k(S)\cong L_kE(S)/L_{k-1}E(S)$ and  $\dim(C_k/C_{k-1})=\dim(\bigoplus\limits_{S\in \Ss}n(S)\cdot \frac{L_kE(S)}{L_{k-1}E(S)})=\sum\limits_{S\in \Ss}n(S)\dim(\alpha^k(S))$ (recall that $n(S)$, respectively $p(T)$, is the multiplicity of the left $C$-comodule $S$, respectively right $C$-comodule $T$, in $C_0$). Therefore, using these we get
\begin{eqnarray*}
\dim(C_{N-1}) & = & \sum\limits_{S\in\Ss}n(S)\dim(L_{N-1}E(S))=\sum\limits_{S\in\Ss}n(S)\sum\limits_{k=0}^{N-1}\dim(\alpha^k(S))
\end{eqnarray*}
Let $T\in\Tt$. If $X=L_kE(\alpha^{-k}(T^*))$ is the left uniserial comodule of length $k+1$ and Loewy length $k$ with simple socle $Y=\alpha^{-k}(T^*)$, 
we have $L_kE(Y)/L_{k-1}E(Y)\cong \alpha^k(Y)\cong \alpha^k(\alpha^{-k}(T^*))=T^*$ as before, 
so the top of $X$ is isomorphic to $T^*$. Hence, $X^*$ is a right comodule with the bottom (socle) isomorphic to $T$, and which is uniserial, so it has Loewy length equal to $k$. Thus $X^*$ embeds in $E(T)$, and it is not difficult to see that this implies that the top $\alpha^{-k}(T^*)^*$ of $X^*$ must occur with positive multiplicity in $L_kE(T)/L_{k-1}E(T)$. Therefore, $\dim(L_kE(T)/L_{k-1}E(T))\geq \dim(\alpha^{-k}(T^*)^*)$. We will show that this inequality is in fact an equality for $k\leq N-1$, and this will imply that $L_kE(T)/L_{k-1}E(T)\cong\alpha^{-k}(T^*)^*$. We have:
\begin{eqnarray*}
\dim(C_{N-1}) & = & \sum\limits_{T\in\Tt}p(T)\sum\limits_{k=0}^{N-1}\dim(L_kE(T)/L_{k-1}E(T))\geq\\
& \geq & \sum\limits_{T\in\Tt}p(T)\sum\limits_{k=0}^{N-1}\dim(\alpha^{-k}(T^*)^*)\\
& = & \sum\limits_{S\in\Ss}n(S)\sum\limits_{k=0}^{N-1}\dim(\alpha^{-k}(S)^*)\,\,\,\,\,\,-\,{\rm since\,}p(T)=n(T^*)\\
& = & \sum\limits_{S\in\Ss}n(S)\sum\limits_{k=0}^{N-1}\dim(\alpha^{-k}(S))\\
& = & \sum\limits_{S\in\Ss}n(S)\sum\limits_{k=0}^{N-1}\dim(\alpha^{k}(S))\\
& = & \dim(C_{N-1})
\end{eqnarray*}
The second to last equality $\sum\limits_{k=0}^{N-1}\dim(\alpha^{-k}(S))=\sum\limits_{k=0}^{N-1}\dim(\alpha^{k}(S))$ is true since $\alpha^N={\rm id}$. The above sequence shows that we must have equalities everywhere, and so the claim is proved. In particular, we obtain that $L_1E(T)/L_0E(T)\cong \alpha^{-1}(T^*)^*$ is simple for all $T$, and by the equivalent characterizations of serial coalgebras we obtain that $C$ is also right serial.
\end{proof}

Combining the results so far, we obtain a precise characterization of coalgebras having the left finite Rat-splitting property:

\begin{theorem}\label{t.SplT}
Let $C$ be a coalgebra. Then $C$ has the left finite Rat-splitting property if and only if $C=\left(\begin{array}{cc} D & M \\ 0 & E \end{array}\right)$ is a generalized triangular comatrix coalgebra, where $D$ is a finite coproduct of (indecomposable) serial coalgebras whose Ext-quiver is a cycle (so $D$ is Artinian), $E$ is a finite dimensional coalgebra and $M$ is a finite dimensional $D$-$E$-bicomodule. 
\end{theorem}
\begin{proof}
If $C$ has the left finite Rat-splitting property, then the decomposition of $C$ as an upper triangular comatrix coalgebra follows from Proposition \ref{p.tr}, and the statement on the coalgebra $D$ follows immediately from Propositions \ref{p.3.6} and \ref{p.3.7}, since the coalgebra $D$ decomposes as a direct sum of components corresponding to the connected components of its Ext-quiver.\\
For the converse, if $D$ is serial then it has the left finite Rat-splitting property as noted in the introduction (\cite{IO}), so we are in the hypothesis of Proposition \ref{p.TrSp}, and the proof is finished.
\end{proof}

The above characterization of coalgebras having the left finite Rat splitting property also allows us to answer the question of whether the left finite Rat-splitting property is equivalent to the right finite Rat-splitting property, asked also in \cite{IO}. 

\section{A few examples}

\begin{example}
Consider the coalgebra $\KK[[X]]^0$, the finite dual of the formal power series algebra. This coalgebra has a basis $c_0,c_1,\dots c_n\dots$ and comultiplication $\Delta(c_n)=\sum\limits_{i+j=n}c_i\otimes c_j$ and counit $\varepsilon(c_n)=\delta_{0,n}$, and is also called the divided power coalgebra (since when ${\rm char}(\KK)=0$, the basis $d_n=\frac{1}{n!}c_n$ has $\Delta(d_n)=\sum\limits_{i=0^n}\left(
\begin{array}{c} n \\ i\end{array} \right)d_i\otimes d_{n-i}$). Then the dual of this coalgebra is $(\KK[[X]]^0)^*\cong \KK[[X]]$ - the algebra of formal power series; the category of rational $\KK[[X]]^*$-modules coincides with that of nilpotent $\KK[X]$ modules, i.e. modules $M$ over the polynomial algebra $\KK[X]$ for which given any $u\in M$, there is $n=n(u)$ such that $X^{n}\cdot u=0$. This is the same as the category of torsion $\KK[[X]]$-modules; obviously, from classical algebra results, then the rational (or torsion) part of every f.g. module splits off, but this is also a particular case of the above result.
\end{example}

\begin{example}
Let $C=\left(\begin{array}{cc} \KK[[X]]^o & \KK \\ 0 & \KK \end{array}\right)$, where $\KK[[X]]^o$ is the divided power coalgebra as above. $\KK$ is a simple left $\KK[[X]]^o$-comodule, and has a trivial right comodule structure over the coalgebra $\KK$; this gives $\KK$ a bicomodule structure. Then $C$ has the left finite Rat-splitting property, but not the right finite Rat-splitting property. Indeed, by Theorem \ref{t.SplT} we have that $C$ has the left finite Rat-splitting property. Denote $x=\left(\begin{array}{cc} 0 & 1 \\ 0 & 0 \end{array}\right)$ and $d=\left(\begin{array}{cc} 0 & 0 \\ 0 & 1 \end{array}\right)$. Then we have $\Delta(d)=d\otimes d$ and $\Delta(x)=c_0\otimes x+x\otimes d$. This shows that the decomposition of $C$ as a direct sum of injective right indecomposable comodules is
$$C=\KK\{x;c_n|n\geq 0\}\oplus \KK\{d\}$$
This shows that the injective hull of the simple (one-dimensional) right $C$-comodule $\KK\{c_0\}$ is $\KK\{x;c_n|n\geq 0\}$ and it has an infinite dimensional subcomodule $\KK\{c_n|n\geq 0\}$. By Theorem \ref{t.1}, $C$ cannot have the right finite Rat-splitting property. 
\end{example}

We finish with two more examples connected to path algebras and path coalgebras of quivers. In general, the dual of a path coalgebra of a quiver is sometimes called the {\it complete} path algebra of that quiver. For connections between these structures we refer the reader to \cite{DIN1}.


\begin{example}\label{e.2}
Let $\tilde{\AA_n}$ be the cyclic oriented quiver with $n$ vertices, and arrows forming an oriented cycle. Let $S$ be a collection of paths in $\AA_n$, and $\overline{S}$ be the set of all paths that are subpath of some path in $S$ (i.e. $\overline{S}$ is the closure of $S$ under taking subpaths). Then the span of $\overline{S}$ is a coalgebra $C$, which is serial; for example, $C$ could be the full path coalgebra of $\tilde{\AA_n}$. This is called a monomial coalgebra. The comodules over $C$ (or the rational $C^*$-modules) are exactly those $C^*$-modules which are nilpotent as modules over the path algebra of $\tilde{\AA_n}$, in the sense of \cite{chin}, or equivalently, modules $M$ for which the annihilator of every element $x\in M$ contains an ideal generated by paths (i.e. a monomial ideal; see also \cite{DIN1}). Then the rational part $Rat(M)$ of a $C^*$-module $M$ coincides with the locally nilpotent part. In fact, it is not difficult to see that in this case $Rat(M)$ is precisely the torsion part of $M$ as it is the case for $\KK[[X]]$. This becomes more clear from the following equivalent description of the structure of the dual $C^*$ of this coalgebra $C$: $C^*$ consists of families $(\alpha_p)_{p\in \overline{S}}$ of elements in $\KK$ indexed by $\overline{S}$, with convolution product similar to that of $\KK[[X]]$, i.e. 
$$(\alpha_p)_{p\in \overline{S}}*(\beta_q)_{q\in\overline{S}}=(\sum\limits_{r=pq}\alpha_p\beta_q)_{r\in\overline{S}}$$
This is usually called the complete path (or monomial) algebra. Moreover, $Rat(M)$ splits off in $M$ for every finitely generated $C^*$-module $M$. Any direct sum $D$ of such coalgebras, and triangular matrix coalgebras constructed with such a $D$ as in Theorem \ref{t.SplT} will have the same property.  
\end{example}

\begin{remark}
Let $C=\left(\begin{array}{cc} D & M \\ 0 & E \end{array} \right)$ be a triangular matrix coalgebra such that:\\
$\bullet$ $D$ is a serial coalgebra obtained as direct sum from serial monomial coalgebras as in Example \ref{e.2} above, (i.e. $D$ is a subcoalgebra of the path coalgebra of a quiver consisting of a disjoint union of cycles, such that $D$ is spanned by paths)\\
$\bullet$ the coalgebra $E$ and the $D$-$E$-bicomodule $M$ are finite dimensional. \\
From Remark \ref{r1} and Proposition \ref{p1}, it follows that for a left $C^*$-module $\left(\begin{array}{c} X \\ Y \end{array} \right)$, the rational part is given by $Rat\left(\begin{array}{c} X \\ Y \end{array} \right) = \left(\begin{array}{c} Rat(X) \\ Y \end{array} \right)$. Thus, the rational $C^*$-modules are exactly those modules for which every element is annihilated by a monomial ideal, where a monomial ideal is an ideal generated by matrices $\left(\begin{array}{cc} p & 0 \\ 0 & 0 \end{array} \right)$ corresponding to paths $p$ in the quiver of $D$. These could still be called (locally) nilpotent modules. Then for every left finitely generated module over the triangular matrix algebra $C^*$ dual to $C$, its rational (or ``locally nilpotent'') part is a direct summand (splits off).
\end{remark}

\begin{example}\label{e.4.6}
Let $Q$ be a finite quiver (finitely many vertices and arrows), and let $C_1,\dots,C_n$ be oriented cycles of $Q$. Let $H$ be a monomial coalgebra of $Q$, i.e. a subcoalgebra of the path coalgebra of $Q$ which has a basis $B$ of paths and assume that:\\
(a) only finitely many elements of $B$ are paths that are not contained in some cycle $C_k$;\\
(b) a path in $B$ that ends at a vertex of a cycle $C_k$ is contained in $C_k$. Equivalently, we may assume that no arrow in $B$ ends at a point in some $C_k$, unless it is in (that) $C_k$, as in example below 

$$\xygraph{ 
!~:{@{-}|@{>}}
a(:@/^-.5pc/[ul]{a_2}:@/^-.5pc/[dl]{\dots}:@/^-.5pc/[dr]{a_k}:[d]{\bullet}"a_k":@/^-.5pc/"a":[r]{\bullet}:[r]{\bullet}[r]e
:[l]"e"
:@/^.5pc/[ur]{e_2}:@/^.5pc/[dr]{\dots}:@/^.5pc/[dl]{e_l}:[d]{\bullet}"e_l":@/^.5pc/"e"[ddd]{g}:[ur]"g"
:@/^-.5pc/[l]{g_2}:[l]{\bullet}"g_2":[u]{\bullet}"g_2":@/^-.5pc/[d]{\dots}:@/^-.5pc/[r]{g_n}:@/^-.5pc/[u]
)}$$

Let $C$ be the subcoalgebra of $H$ having as basis the subpaths of cycles $C_k$, $k=1,\dots,n$, and let $Y$ be the subspace spanned by the set $U$ of the rest of the paths in $B$. It is not difficult to see that the combinatorial condition (b) translates to the fact that $Y$ is a left subcomodule in $H$ (since if a path $p$ is in $U$, then all the paths $s$ for which $p=rs$ are in $U$). Also, condition (a) translates into the fact that $Y$ is finite dimensional. Now using Proposition \ref{p.triang}, we see that $H$ is an upper triangular matrix coalgebra, with the left corner isomorphic to $C$, and the second column comodule isomorphic to $Y$. Moreover, $C$ is serial (note that (b) implies that the $C_k$'s are disjoint). Hence, $H$ satisfies the conditions of Theorem \ref{t.SplT}, and has the left finite Rat splitting property.  \\
Note again that the algebra $C^*$ is in fact a complete monomial algebra obtained from the quiver $Q$ and the paths $B$; that is the structure of $A=C^*$ is given by
\begin{eqnarray*}
C^* & = & \{(\alpha_p)_{p\in B}| \alpha_p\in \KK\} \\
(\alpha_p)_{p\in B} * (\beta_{q})_{q\in B} & = & (\sum\limits_{pq=r}\alpha_p\beta_q)_{r\in B} 
\end{eqnarray*}
Moreover, by the above, the locally nilpotent part of every left finitely generated left $A$-module is a direct summand; nevertheless, if $U\neq \emptyset$, the locally nilpotent part of finitely generated right $A$-modules is not always a direct summand.
\end{example}


\bigskip\bigskip

\begin{center}
\sc Acknowledgment
\end{center}
This work written within the frame of the strategic grant POSDRU/89/1.5/S/58852, Project ``Postdoctoral programe for training scientific researchers'' cofinanced by the European Social Fund within the Sectorial Operational Program Human Resources Development 2007-2013. The research was also supported by the UEFISCDI Grant PN-II-ID-PCE-2011-3-0635, contract no. 253/5.10.2011 of CNCSIS. The author acknowledges the work of the referees whose detailed reports helped improve this paper.


\vspace*{3mm} 
\begin{flushright}
\begin{minipage}{148mm}\sc\footnotesize

Miodrag Cristian Iovanov\\
University of Iowa\\
Department of Mathematics, McLean Hall \\
Iowa City, IA, USA\\
and\\
University of Bucharest, Faculty of Mathematics, Str. Academiei 14\\ 
RO-010014, Bucharest, Romania\\
{\it E--mail address}: {\tt
yovanov@gmail.com; miodrag-iovanov@uiowa.edu}\vspace*{3mm}

\end{minipage}
\end{flushright}
\end{document}